\documentclass[letterpaper, 10 pt, conference]{ieeeconf}  % Comment this line out
                                                       
\IEEEoverridecommandlockouts                              % This command is only
\usepackage{amsmath} % assumes amsmath package installed
\usepackage[nolist]{acronym}
\usepackage{xcolor}
\usepackage{amsmath}
\usepackage{amssymb}
\usepackage{mathtools}
\usepackage{subcaption}
\usepackage{hyperref}
\hypersetup{colorlinks=true, linkcolor=black}
\usepackage[ruled]{algorithm2e}
\usepackage{mathrsfs}
\definecolor{ccolor}{RGB}{203,96,21}

\definecolor{purplecustom}{RGB}{150, 0, 255}
\newcommand{\re}[1]{#1}
\newcommand{\rev}[1]{#1}
\newcommand{\R}{\mathbb{R}} 
\newcommand{\K}{\mathbb{K}}

\newcommand{\ls}{\mathbb{L}}
\newcommand{\dc}{\mathcal{D}}

\newcommand{\0}{\mathbf{0}}

\newcommand{\psd}{\mathbb{S}}

\newcommand{\dx}{\Delta x} 
\newcommand{\du}{\Delta u}

\newcommand{\tr}{\top}

\newcommand{\kron}{\otimes}

\newcommand{\pc}{\mathcal{P}}
\newcommand{\bpc}{\bar{\mathcal{P}}}

\DeclarePairedDelimiter{\abs}{\lvert}{\rvert}
\DeclarePairedDelimiter{\norm}{\lVert}{\rVert}

\DeclarePairedDelimiter{\diag}{\textrm{diag}(}{)}

\DeclarePairedDelimiterX{\inp}[2]{\langle}{\rangle}{#1, #2}

\DeclareMathOperator*{\find}{find}

\newtheorem{thm}{Theorem}[section]

\newtheorem{prob}{Problem}

\newtheorem{rmk}{Remark} 
\usepackage{color}
\usepackage{ulem}

\title{\LARGE \bf Data-Driven Superstabilizing Control under Quadratically-Bounded Errors-in-Variables Noise
}

\author{Jared Miller$^{1,3}$, Tianyu Dai$^2$, Mario Sznaier$^3$
\thanks{$^1$ J. Miller is with the Automatic Control Laboratory (IfA) and NCCR Automation, Department of Information Technology and Electrical Engineering (D-ITET), ETH Z\"{u}rich, Physikstrasse 3, 8092, Z\"{u}rich, Switzerland (jarmiller@control.ee.ethz.ch).}
\thanks{$^2$ T. Dai is with The MathWorks, Inc., 1 Apple Hill Drive,
		Natick, MA 01760 USA (tdai@mathworks.com).}
\thanks{$^3$ M. Sznaier is with the Robust Systems Lab,  ECE Department, Northeastern University, Boston, MA 02115. (msznaier@coe.neu.edu).}
\thanks{J. Miller was partially supported by the Swiss National Science Foundation under NCCR Automation, grant agreement 51NF40\_180545.  J. Miller and M. Sznaier were partially supported by NSF grants  CNS--2038493, AFOSR grant FA9550-19-1-0005, and ONR grant N00014-21-1-2431.  
}}

\begin{document}

\maketitle
\thispagestyle{empty}
\pagestyle{empty}

%%%%%%%%%%%%%%%%%%%%%%%%%%%%%%%%%%%%%%%%%%%%%%%%%%%%%%%%%%%%%%

\begin{abstract}
\label{sec:abstract}
% This paper performs full-state-feedback stabilizing control for a class of linear systems consistent with input-noise and state-noise corrupted data. 
The Error-in-Variables model of system identification/control involves nontrivial input and measurement corruption of observed data, resulting in generically nonconvex optimization problems. This paper performs full-state-feedback stabilizing control of all discrete-time linear systems that are consistent with observed data for which the input and measurement noise obey quadratic bounds. 
Instances of such quadratic bounds include elementwise norm  bounds (at each time sample), energy bounds (across the entire signal), and chance constraints arising from (sub)gaussian noise.
Superstabilizing controllers are generated through the solution of a sum-of-squares hierarchy of semidefinite programs. A theorem of alternatives is employed to eliminate the input and measurement noise process, thus improving tractability\re{.} 
% Effectiveness of the scheme is generated on an example system in the chance-constrained set-membership setting where the input and state-measurement noise are i.i.d. normally distributed.

% \urg{Abstract goes here}

\end{abstract}
\section{Introduction}
\label{sec:introduction}

This paper proposes a method to stabilize linear systems corrupted by quadratically bounded \ac{EIV} noise \cite{soderstrom2018errors}. State and input observations $\dc = \{\hat{x}_t, \hat{u}_t\}_{t=\re{1}}^T$ are collected along a trajectory of an \re{a-priori unknown} $n$-state $m$-input linear system \re{$(A_{\re{\star}}, B_{\re{\star}})$}:
\begin{subequations}
\label{eq:model}
\begin{align}
    x_{t+1} &= A_{\re{\star}} x_t + B_{\re{\star}} u_t \label{eq:discrete_dynamics}\\
    \hat{x}_t &= x_t + \dx_t,
    \quad   \hat{u}_t = u_t + \Delta u_t,
    \label{eq:noise_corrupt}
\end{align}
\end{subequations}
in which the collected data in $\dc$ is corrupted by state noise $\dx \in \R^{n T}$ and input noise $\du \in \R^{m(T-1)}$. This paper focuses on the setting where the noise processes $(\dx, \du)$ satisfy a collection of $\re{L}$ convex quadratic constraints as
\begin{align}
\label{eq:quad_bound}
    \forall \ell & \in 1..L: & \norm{F_\ell \dx + G_\ell \du}_2 \leq 1,
\end{align}
in which the known constraint matrices $F_j$ and $G_j$ have compatible dimensions. These quadratic constraints could arise from deterministic knowledge of $(\dx, \du)$, or from high-probability chance-constraints imposed on stochastic processes $(\dx, \du)$ if a robust description is overly conservative \cite{ben2009robust, MARTIN2023gaussian}.
% ($\forall j \in 1..N, \ \exists s_j \in \N \mid F_j \in \R^{s_j \times nT}, \ G_j \in \R^{s_j \times m(T-1)}$). 
Our goal is to find a gain matrix $K$ such that the full-state-feedback control policy $u_t = K x_t$ can simultaneously stabilize all plants that are consistent with the data in $\dc$ under the noise description in \eqref{eq:quad_bound}.

% full-state-feedback controller gain matrix $K$

% \urg{Introduction goes here. }

This paper follows the framework of set-membership direct \ac{DDC}.
In direct \ac{DDC}, a control policy is formed from the collected data and modeling assumptions without first performing system identification (and synthesizing a controller for the identified system) \cite{HOU20133}. Set-membership \ac{DDC} has three main ingredients: the set of data-consistent plants (given a noise model), the set of commonly stabilized plants by a designed controller, and the certificate of set-containment that the stabilized-set contains the consistent-set \cite{martin2023guarantees}. The Matrix S-Lemma can be used to provide proofs of quadratic and worst-case $H_2$ or $H_\infty$ robust control when the noise model is defined by a matrix ellipsoid (quadratic matrix inequality) \cite{waarde2020noisy,bisoffi2021trade}. Farkas-based certificates for polytope-in-polytope membership have been used for robust superstabilization \cite{dai2018data} and positive-stabilization \cite{miller2023ddcpos}. \ac{SOS} certificates of nonnegativity have been employed for stabilization of more general nonlinear systems \cite{dai2020semi, martin2021data}. We note that other non-set-membership \ac{DDC} methods include using Virtual Reference Feedback Tuning \cite{campi2002virtual} and Willem's Fundamental Lemma \cite{coulson2019data}. 

% Farkas certificates for polytope-in-polytope containment \cite{cheng2015robust, miller2023ddcpos}, a Matrix S-Lemma for \acp{QMI} to prove quadratic and robust stabilization \cite{waarde2020noisy, van2023quadratic, miller2022lpvqmi}, and Sum-of-Squares certificates of polynomial nonnegativity \cite{dai2020semi, martin2021data, miller2022eiv_short, zheng2023robust}.

% \urg{Literature review here. Survey of set-membership \ac{DDC}. Also note other instances of \ac{DDC}}.

Most \ac{DDC} methods focus solely on process-noise corruption, allowing for the synthesis of controllers through the solution of computationally simple convex programs. This paper continues a line of work in addressing the more challenging setting of \ac{EIV} superstabilization,
proposing a method that can handle the more difficult but realistic \ac{EIV} noise scenario at the cost of more expensive computational requirements. 
Prior work on superstabilization \cite{SZNAIER19963550, polyak2002superstable} of \ac{EIV} systems includes full-state-feedback for polytope-bounded noise \cite{miller2022eiv_short} (e.g. $L_\infty$-bounds) and dynamic output feedback for SISO plants \cite{miller2023eivarx}.
% In these previous investigations, we formulated infinite-dimensional \acp{LP} that would certify  (extended) superstabilization of the control policy $u = Kx$ \cite{SZNAIER19963550, polyak2002superstable, polyak2004extended}. 
% These \acp{LP} were truncated using the Moment-\ac{SOS} hierarchy of \acp{SDP} \cite{lasserre2009moments}, with a guarantee of recovering the superstabilizing controller (if it exists) as the polynomial degree tends towards infinity. Theorems of Alternatives were used to eliminate the noise variables $(\dx, \du)$ \cite{boyd2004convex} from the \acp{LP}, yielding more tractable \acp{SDP} without introducing any conservatism.
In this work, we will ensure superstabilization under quadratically bounded noise. This will involve developing matrix-\ac{SOS} expressions defined for (multiple) quadratic constraints 
 in \eqref{eq:quad_bound}. Computational complexity is reduced by eliminating $(\dx, \du)$ using a Theorem of Alternatives \cite{boyd2004convex}. 
The concurrent and independently developed similar work in \cite{bisoffi2024controller} performs lossless quadratic stabilization in the presence of a single quadratic-matrix-inequality-representable quadratic constraint, with conservatism added under multiple quadratic constraints. 
% This work focuses on the superstabilization setting with noise sets described by multiple quadratic constraints.
% express the quadratic constraints in \eqref{eq:quad_bound} through 
% form convex programs based on 

% \input{sections/contributions}

This \re{letter} has the following structure: 
Section \ref{sec:preliminaries} reviews preliminaries including notation, quadratically constrained noise, superstabilization, and matrix-\ac{SOS} methods. Section \ref{sec:quad_lp} presents infinite-dimensional \acp{LP} to perform superstabilization under quadratically-bounded \ac{EIV} noise, and applies a Theorem of Alternative to eliminate the noise variables $(\dx, \du)$ from the linear inequalities. Section \ref{sec:quad_truncate} truncates the infinite-dimensional \acp{LP} using the moment-\ac{SOS} hierarchy of \acp{SDP}, and tabulates computational complexity of different approaches. 
% Section \ref{sec:extensions} incorporates recorded process noise into the quadratically-bounded framework.
\re{Section \ref{sec:extended_super} describes extended superstabilization \cite{polyak2004extended} of data-consistent systems.}
Section \ref{sec:examples} demonstrates our method for \ac{EIV}-tolerant superstabilizing control of example systems. Section \ref{sec:conclusion} concludes the \re{letter}.

% \urg{Fill in the paper structure}
% Section \ref{sec:preliminaries} will review preliminaries such as notation, notions of stability for linear systems, and \ac{SOS} proofs of polynomial nonnegativity. Section \ref{sec:full_method} will present 
% The paper is concluded in Section \ref{sec:conclusion}.
\section{Preliminaries}
\label{sec:preliminaries}

% \subsection{Acronyms/Initialisms}
\begin{acronym}[WSOS]
\acro{BSA}{Basic Semialgebraic}
\acro{DDC}{Data Driven Control}

\acro{EIV}{Errors-in-Variables}

% \acro{GAS}{Globally Asymptotically Stable}

% \acro{CSP}{Correlative Sparsity Pattern}

\acro{LP}{Linear Program}
\acroindefinite{LP}{an}{a}

\acro{LMI}{Linear Matrix Inequality}
\acroplural{LMI}[LMIs]{Linear Matrix Inequalities}
\acroindefinite{LMI}{an}{a}

% \acro{LQR}{Linear Quadratic Regulator}
% \acroindefinite{LQR}{an}{a}

% \acro{LP}{Linear Program}
% \acroindefinite{LP}{an}{a}
% \acro{OCP}{Optimal Control Problem}

% \acro{ODE}{Ordinary Differential Equation}

% \acro{POP}{Polynomial Optimization Problem}

\acro{PMI}{Polynomial Matrix Inequality}

\acro{PSD}{Positive Semidefinite}

% \acro{PD}{Positive Definite}

% \acro{PDE}{Partial Differential Equation}

\acro{SDP}{Semidefinite Program}
\acroindefinite{SDP}{an}{a}

\acro{SOC}{Second-Order Cone}
\acroindefinite{SOC}{an}{a}

\acro{SOS}{Sum-of-Squares}
\acroindefinite{SOS}{an}{a}

\acro{WSOS}{Weighted Sum of Squares}

\end{acronym}

\subsection{Notation}

\begin{tabular}{p{0.15\columnwidth}p{0.75 \columnwidth}}
     $a..b$& natural numbers between $a$ and $b$ (inclusive) \\
     $\R_{\geq 0} \ \re{(\R_{>0})}$ & nonnegative \re{(positive)} orthant \\
     $\pi^x$ & projection operator $(x, y) \rightarrow x$ \\
     $A^\tr$ & Transpose of matrix $A$ \\
      $A^+$ & Pseudoinverse of matrix $A$ \\
     $\mathbb{L}^n$ & $n$-dimensional \ac{SOC} \\
     & $\{(x, t) \in \R^n \times \R_{\geq 0} \mid t \geq \norm{x}_2\}$ \\
     $\mathcal{N}(\mathbf{m}, \Sigma)$& Normal distribution with mean $\mathbf{m}$ \\
     & and covariance $\Sigma$ \\
     $\R[x]$ & polynomials in indeterminate $x$ \\
     $\deg(h)$ & degree of polynomial $h \in \R[x]$ \\
     $\psd^r[x]$ & $r \times r$ polynomial-valued-matrices in $x$ \\
     $\Sigma^r[x]$ & \ac{SOS} polynomial-valued-matrices in $x$ \\
     $\inp{A}{B}$ & Matrix pairing $\textrm{Tr}(A^\tr B) = \sum_{ij} A_{ij}B_{ij}$
\end{tabular}

\subsection{Quadratic Noise Bounds}

This subsection briefly highlights instances of quadratic noise bounds in \eqref{eq:quad_bound}.

\subsubsection{Elementwise Norm Constraints}
\label{sec:elem_norm}
A deterministic noise bound could impose that $\forall t \in 1..T: \norm{\dx_t}_2 \leq \epsilon_x$, and  $\forall t \in 1..T-1: \norm{\du_t}_2 \leq \epsilon_u$. 
% Matrix scaling could also be added, such as in $\norm{F_t \dx_t}_2 \leq \epsilon_x$.
Elementwise norm constraints can arise when $\dx$ and $\du$ are i.i.d. normally distributed random variables $(\dx_t  \sim \mathcal{N}(\0_n, \Sigma_x), \du_t \sim \mathcal{N}(\0_m, \Sigma_u)$. The statistics $\dx_t^T \Sigma_x^{-1} \dx_t$ and $\du_t^T \Sigma_u^{-1} \dx_t$  are each $\chi^2$-distributed with $n$ and $m$ degrees of freedom respectively.

Let $\varepsilon(\delta; f)$ refer to the quantile statistic of a $\chi^2$ distribution with $f$ degrees of freedom with respect to a probability $\delta \in [0, 1]$ (and random variable $p$):
\begin{equation}
    \chi^2_f(p \leq \varepsilon(\delta; f)) = 1-\delta.
\end{equation}

For a choice of probabilities $\delta_x, \delta_u \in [0, 1]$, the joint probability of $\dx, \du$ lying within the set described by
\begin{subequations}
\label{eq:chisq}
\begin{align}
    \forall t & \in 1..T: & &  \norm{\Sigma_x^{-1/2} \dx_t}_2 & \leq \sqrt{\varepsilon(\delta_x; n)} 
    \label{eq:chisq_x}\\
    \forall t & \in 1..T-1: & &  \norm{\Sigma_u^{-1/2} \du_t}_2 & \leq \sqrt{\varepsilon(\delta_u; m)} \label{eq:chisq_u}
\end{align}
\end{subequations}
is $P_{\textrm{joint}} = (1-\delta_x)^T(1-\delta_u)^{T-1}$. A controller $u = K x$ certifiably stabilizes all plants consistent with \eqref{eq:model} and \eqref{eq:chisq} will be able to stabilize the true system with probability $P_{\textrm{joint}}$. Similar elementwise quadratic chance-constraints arise when $\dx, \du$ are drawn from i.i.d. subgaussian distributions \cite{hsu2012subgaussian}.

\subsubsection{Energy Constraints}
The standard quadratic expression used in a Linear Quadratic Regulator of
\begin{equation}
    \textstyle J = \dx_T^\tr Q_T \dx_T + \left(\sum_{t=1}^{T-1} \dx_t^\tr Q \dx_t + \du_t^\tr R \du_t\right)
\end{equation}
is compatible with the structure of \eqref{eq:quad_bound} if \re{the cost matrices} $(Q, R, Q_T)$ are all \ac{PSD}\re{.} 
% The mixed structure (for appropriately dimensioned $\tilde{Q}, \tilde{R}, \tilde{S}$)
% \begin{equation}
%     J_{\textrm{mix}} = \begin{bmatrix}
%         \dx \\ \du
%     \end{bmatrix}^T \begin{bmatrix}
%         \tilde{Q} & \tilde{S} \\
%         \tilde{S}^\tr & \tilde{R}
%     \end{bmatrix} \begin{bmatrix}
%         \dx \\ \du
%     \end{bmatrix}
% \end{equation}
% can also be posed as the left-hand-side of the \eqref{eq:quad_bound} inequality, thus allowing for energy constraints on the noise $(\dx, \du)$.

\subsection{Superstabilizing Control}
% \urg{Hurwitz, Quadratic Stability, Superstability. Link superstability with the Gershgorin Circle Criterion}

Let $W \in \R^{f \times n}$ be a matrix with full column rank such that $\{x \mid \norm{W x}_\infty \leq 1\}$ is a compact set. 
A discrete-time linear system $x_{t+1} = A x_t$ is \textit{$W$-superstable} if $\norm{W x}_\infty$ is a polyhedral Lyapunov function:
\begin{equation}
    \norm{W A W^{\re{+}}}_\infty < 1 \quad \text{($L_\infty$ Operator Norm).}  \label{eq:ext_superstability_norm}
\end{equation}
The system is \textit{superstable} if it is $W$-superstable with $W = I_n$. Any superstable system $x_{t+1} = A x_t$  obeys the the decay bound of $\norm{x_t}_\infty \leq \norm{A}_\infty^t \norm{x_0}_\infty$. This decay bound generalizes to $W$-superstability as in $\norm{W x_t}_\infty  \leq \norm{W A W^{\re{+}}}_\infty^t \norm{W x_0}_\infty$.
$W$-superstabilization of the system $x_{t+1} = A x_t + B u_t$ proceeds by choosing a gain $K \in \R^{m \times n}$ with $u_t = K x_t$ such that $\norm{\re{W}(A + BK)\re{W^+}}_\infty < 1$. The $W$-superstabilization problem of minimizing the decay bound \re{(for fixed $W$)} is a finite-dimensional \ac{LP}: 
\begin{subequations}
\label{eq:W_superstable}
\begin{align}
    \lambda^* = & \inf_{M, K}\quad \lambda: \qquad \textstyle \sum_{j=1}^f M_{ij} < \lambda & & \forall i \in 1..f \label{eq:superstable_strict} \\
    &\textstyle  \abs{[W(A +  B K)W^+]_{ij}}\leq M_{ij}& & \forall i,j \in 1..f.\label{eq:superstable_nonstrict} \\
    & M \in \R^{n \times n}, \  K \in \R^{m \times n}.
\end{align}
\end{subequations}

% Problem \eqref{eq:ext_superstable} performs optimization w.r.t. a fixed matrix $W$. Extended superstability \cite{polyak2004extended} optimizes over the entries of a positive-diagonal matrix $W$, but computing the peak-to-peak gain in this extended setting requires the solution of a parametric \ac{LP} with a single free parameter.

\subsection{Sum-of-Squares Matrices Background}
\label{sec:sos}
This paper will formulate worst-case-superstabilization problems as infinite-dimensional \acp{LP}, which in turn will be truncated using the moment-\ac{SOS} hierarchy \cite{lasserre2009moments}.
For any symmetric polynomial-valued-matrix $p \in \psd^r[x]$ of size $r \times r$ with indeterminate $x \in \R^n$, the degree of $p$ is the maximum polynomial degree of any one of its entries $(\deg p = \max_{ij} \deg p_{ij}).$ A sufficient condition for $p(x) \succeq 0$ over $\R^n$ is if $p(x)$ is an \textit{\ac{SOS}-matrix}: there exists a polynomial vector $v \in (\R[x])^c$ and a \ac{PSD} \textit{Gram} matrix $Q \in \psd^{rc \times rc}$ such that (Lemma 1 of \cite{scherer2006matrix})
\begin{align}
    p(x) = (v(x) \otimes I_r)^\tr Q (v(x) \otimes I_r). \label{eq:sos_mat}
\end{align}
The set of \ac{SOS} matrices with representation in \eqref{eq:sos_mat} is $\Sigma^r[x]$, and the subset of \ac{SOS} matrices with maximal degree $\leq 2k$ is $\Sigma^r[x]_{\leq 2k}$.
A constraint region defined by a locus of \ac{PSD} containments can be constructed from $\{g_j \in \psd^{r_j}[x]\}_{j=1}^{N_c}$ as
\begin{align}
    \K = \{x \in \R^n \mid \forall j \in 1..N_c: \ g_j(x) \succeq 0\}. \label{eq:set_K}
\end{align}
The matrix  $p \in \psd^r[x]$ satisfies \iac{PMI} over the region in \eqref{eq:set_K} if $\forall x \in \K: \ p(x) \succeq 0. $
% \begin{align}
    % \label{eq:pmi_K}
% \end{align}
A sufficient condition for this \ac{PMI} to hold is that there exist \ac{SOS}-matrices $\{\sigma_j(x)\}_{j=0}^{N_c}$ \re{and an $\epsilon > 0$} such that
\begin{subequations}
\label{eq:scherer_psatz}
\begin{align}
    p(x) &= \sigma_0(x) + \textstyle \sum_{i=j}^{N_c} \inp{g_j(x)}{\sigma_j(x)} \re{+ \epsilon I}\label{eq:scherer_psatz_expr}\\
    & \sigma_0 \in \Sigma^r[x], \ \forall j \in 1..N_c: \  \sigma_j \in \Sigma^{r_j}[x] \label{eq:scherer_multipliers}, \epsilon > 0.
\end{align}
\end{subequations}
The set of matrices in $\psd^r[x]$ possessing a representation as in \eqref{eq:scherer_psatz_expr} (existence of $\{\sigma_j\}$) will be written as the \ac{WSOS} set $\Sigma^r[\K]$. The degree-$2k$-bounded \ac{WSOS} cone $\Sigma^r[\K]_{\leq 2k}$ is the set of matrices with representation in \eqref{eq:scherer_psatz} such that $\deg \sigma_0 \leq 2k$ and $\forall j \in 1..N_c: \ \deg \inp{g_j(x)}{\sigma_j(x)} \leq 2k$. The representation of $\K$ by polynomial matrices in \eqref{eq:set_K} is \textit{Archimedean} if there exists an $R > 0$ such that the scalar polynomial $p_R(x) = R - \norm{x}_2^2$ satisfies $p_R \in \Sigma^r[\K]$. If the representation in \eqref{eq:set_K}  is Archimedean, then every $p \in \psd^r[x]$ such that $p(x)$ is Positive Definite over $K$ satisfies $(p(x) - \epsilon I_r ) \in \Sigma^r[K]$ for some $\epsilon > 0$ (Corollary 1 of \cite{scherer2006matrix}).
Testing membership of $p \in \Sigma^r[x]_{\leq 2k}$ through \eqref{eq:scherer_psatz} can be accomplished by solving \iac{LMI} in the Gram matrices for $\{\sigma_{j}\}$ under coefficient-matching equality constraints. 
% \urg{Talk about the Scherer Psatz, matrix \ac{SOS} verification.}
A common choice of polynomial vector $v(x)$ used to represent the \ac{SOS} matrices in \eqref{eq:sos_mat} is the vector of all $\binom{n+k}{k}$ monomials of degree $\leq k$. The maximal-size \ac{PSD} constraint involved in \eqref{eq:scherer_psatz} under the monomial choice for $v(x)$ is $r \binom{n+k}{k}$.

\section{Quadratically-Bounded Linear Programs}
\label{sec:quad_lp}
% \urg{The main theory}

This section presents an infinite-dimensional \ac{LP} which $W$-superstabilizes  the class of data-consistent plants.

\subsection{Consistency Set}

Define $h^0\re{(A, B)}$ as the following residual:
  \begin{align}
        h_t^0 \re{(A, B)} &= \hat{x}_{t+1} - A \hat{x}_t - B \hat{u}_t & \forall t \in 1..T-1. \label{eq:aff_weight}
    \end{align}
The quantity $\re{h^0_t(A_\star, B_\star)}$  will be equal to zero at all $t \in 1..T-1$ if there is no noise present in the system.
% , in which the data records in $\dc$ exactly track the dynamics.   
The joint consistency set $\bpc(A, B, \dx, \du) \subset \R^{n \times n} \times \R^{n \times m} \times \R^{nT} \times \R^{m(T-1)}$ of plants and noise values consistent with the dynamics \eqref{eq:discrete_dynamics} and noise properties \eqref{eq:quad_bound} is
\begin{align}
    \bar{\pc}:= \ \begin{Bmatrix*}[l]\dx_{t+1}= A\dx_{t} + B \du_t + h_t^0  & \forall t \in 1..T-1 \\
       \norm{F_\ell \dx + G_\ell \du}_2 \leq 1  & \forall \ell  \in 1..L\\
     \end{Bmatrix*}, \label{eq:pcbar}
\end{align}
\re{in which the $(A, B)$ dependence of $h^0_t$ is implied in notation.}
The formulation \eqref{eq:pcbar} involves bilinear terms $A \dx$ and $B \du$, which could lead to nonconvex or possibly disconnected sets $\bpc$ \cite{cerone1993feasible}.
By defining a matrix $\Xi_A \in \R^{n(T-1) \times nT}$ with
\begin{align}
    % \Xi &= \begin{bmatrix}
    %     A & -I & \0 & \hdots \\
    %     \0 & A & -I & \hdots \\
    %     \vdots & \vdots & \vdots & \ddots
    % \end{bmatrix} \\
    \Xi &= [(I_{T-1} \otimes A), \0_{n(T-1) \times n}] + [\0_{n(T-1) \times n}, -I_{n(T-1)}] ,\nonumber 
\end{align}
and defining $n_\ell$ to be the column dimension of $F_\ell$ and $G_\ell$ for each $\ell$ ($F_\ell \in \R^{n_\ell \times nT}, \ G_\ell \in \R^{n_\ell \times m(T-1)}$\re{)},
the consistency set in \eqref{eq:pcbar} can be equivalently expressed as
\begin{align}
    \bar{\pc}:= \ \begin{Bmatrix*}[l]\Xi \dx + (I_{T-1} \otimes B) \du + h^{\re{0}} = 0  \\
       (F_\ell \dx + G_\ell \du, 1)  \in \ls^{n_\ell}  & \forall \ell  \in 1..L\\
     \end{Bmatrix*}. \label{eq:pcbar_cone}
\end{align}

% % \begin{bmatrix}
%         A & -I & 0 & \ldots \\
%         & A & -I
%     \end{bmatrix}

The consistency set $\pc(A, B) \subset \R^{n \times n} \times \R^{n \times m}$ of plants compatible with the data in $\dc$ is the following projection,
\begin{align}
    \pc(A, B) = \pi^{A,B} \bpc(A, B, \dx, \du),\label{eq:pc}
\end{align}
in which there exists bounded noise $(\dx, \du)$ following \eqref{eq:quad_bound} such that $\dc$ could have been observed.
Our task is as follows:
\begin{prob}
\label{prob:ess}
    Find a matrix $K \in \R^{m \times n}$ such that full-state-feedback controller $u_t = K x_t$ ensures that $A + BK$ is $W$-superstable for all $(A, B) \in \pc$.
\end{prob}

% \begin{rmk}
%     % The matrix $K$ should be chosen to 
% \end{rmk}

\subsection{Full Linear Program}

Let $\eta > 0$ be chosen as a margin to ensure stability.

\begin{thm}
$W$-Superstability through a common $K$ (for Problem \ref{prob:ess}) can be ensured by solving the following infinite-dimensional \ac{LP} in terms of a matrix $K \in \R^{m \times n}$ and a matrix-valued function $M(A, B): \pc \rightarrow \R^{f \times f}$ (Equation (12) of \cite{miller2022eivlong}):
\begin{subequations}
\label{eq:super_full} 
\begin{align}
    \find_{M, K} \ \textrm{\re{s.t.}} \ & \forall (A, B, \dx, \du) \in \re{\bpc}:  \label{eq:super_full_quant}\\
    & \quad \forall i = 1..f: \label{eq:super_full_row}\\
    & \qquad 1 - \eta - \textstyle\sum_{j=1}^n M_{ij}(A, B) \geq 0 \nonumber  \\        
    & \quad \forall i = 1..f, \ j = 1..f: \label{eq:super_full_element}\\
    & \qquad  M_{ij}(A, B) -(W\left(A+BK\right)W^+)_{ij} \geq 0 \nonumber\\
    & \qquad  M_{ij}(A, B) +(W\left(A+BK\right)W^+)_{ij} \geq 0\nonumber \\
    & M(A, B): \pc \rightarrow \R^{f \times f}, \ K \in \R^{m \times n}.\label{eq:M_function}
\end{align}
\end{subequations}
\end{thm}
\begin{proof}
The $W$-superstabilizing linear system in \eqref{eq:W_superstable} is fulfilled for each $(A, B) \in \pc$, as expressed by the $\forall (A, B, \dx, \du) \in \bpc$ quantification in \eqref{eq:super_full_quant}.
\end{proof}

\begin{rmk}
    \re{Under} the assumption that $\bpc$ is compact, the function $M$ can be selected to be continuous (Lemma 3.1 of \cite{miller2022eivlong}) and even polynomial (Lemma 3.2 of \cite{miller2022eivlong}).
% \end{rmk}
% \begin{rmk}
    % This follows from Lemmas 3.1 (continuity) and 3.2 (polynomial approximability) of \cite{miller2022eivlong}. 
    The same type of linear system structure for superstabilization in \eqref{eq:super_full_row}-\eqref{eq:M_function} is used (with $W=I_n$) as in \cite{miller2022eivlong}, but the set $\bpc$ in \cite{miller2022eivlong} is defined by polytopic-bounded noise rather than  quadratically-bounded noise  \eqref{eq:pcbar}. 
    \end{rmk}
    % Proof details are omitted for space.
% \end{proof}
% \urg{LCSS doesn't like these external references, but I'm not sure what to do}.

\subsection{Robustified Linear Program}

Problem \eqref{eq:super_full} involves \re{\iac{LP}} with $2f^2 + f$ infinite-dimensional \re{linear} constraints, each posed over the $n(n+T) + m(n+T-1)$ variables $(A, B, \dx, \du)$.
Reducing the number of variables involved in any quantification will simplify resultant \ac{SOS} truncations and increase computational tractability.
% Given that the complexity of finite-dimensional discretizations for infinite-dimensional \acp{LP} generally scales in an exponential number with the number of variables (curse of dimensionality), it is vital to try and pose problems that can reduce the maximum number of variables appearing in any infinite-dimensional linear inequality constraint. 
This subsection will use a Theorem of Alternatives in order to eliminate the noise variables $(\dx, \du)$ from Program \eqref{eq:super_full}. The maximum number of variables appearing in any infinite-dimensional linear inequality constraint will subsequently drop from $n(n+T) + m(n+T-1)$ down to $n(n+m)$,  \re{which no longer depends on the number of samples $T$.}

Let $q(A, B): \pc \rightarrow \R$ be a function that is independent of $(\dx, \du)$, such as any one of the left-hand-side elements for  constraints \eqref{eq:super_full_row}-\eqref{eq:super_full_element}. The following pair of problems are strong alternatives (exactly one is feasible):
\begin{subequations}
\label{eq:strong_altern}
\begin{align}
    \forall (A, B, \dx, \du) \in \bpc: & & q(A, B) &\geq 0 \label{eq:q_feas} \\
    \exists  (A, B, \dx, \du) \in \bpc: & & q(A, B) &< 0, \label{eq:q_infeas}
\end{align}
\end{subequations}
because if $q\geq 0$ for all $(A, B, \dx, \du) \in \bpc$ \eqref{eq:q_feas}, then there cannot be a point in $\bpc$ where $q(A, B, \dx, \du) < 0$ \eqref{eq:q_infeas}.
We can define the following dual variable functions
\begin{subequations}
\label{eq:dual_must}
    \begin{align}
        \mu: \ & \pc \rightarrow \R^{n(T-1)} \\
        \label{eq:soc_cond}  (s_\ell, \tau_\ell): \ & \pc \rightarrow \ls^{n_\ell} & \forall \ell \in 1..L.
        % s_\ell: \ & \pc \rightarrow \R^{n_\ell}, \ \tau_\ell: \  \pc \rightarrow \R_{\geq 0} & & \forall \ell \in 1..L.
    \end{align}
\end{subequations}
% under the \ac{SOC} constraint that
% \begin{equation}
% \label{eq:soc_cond}
    % \forall \ell \in 1..L, \ (A, B) \in \pc: (s_\ell(A, B), \tau_\ell(A, B)) \in \ls^{n_\ell}.
% \end{equation}

The $(A, B)$ dependence of $(\mu, s, \tau)$ will be omitted to simplify notation.
A weighted sum $\Phi(A, B, \dx, \du; \mu, s, \tau)$ may be constructed from $q$ and  $(\mu, s, \tau)$ from \eqref{eq:dual_must}:
\begin{align}
    \Phi &= -q + \mu^\tr(\Xi \dx + (I_{T-1} \otimes B) \du + h^0) \nonumber  \\ &\textstyle+ \sum_{\ell=1}^L \left(\re{\tau_\ell\ - } s_\ell^\tr(F_\ell \dx + G_\ell \du) \right). \label{eq:phi_orig}
\end{align}
The terms in $\Phi$ may be rearranged into 
\begin{align}
    \Phi &= \textstyle \left(-q + \mu^\tr h^0\re{+} \sum_{\ell=1}^L \tau_\ell\right) \nonumber \\
    &\textstyle + \left(\Xi^\tr \mu \re{-} \sum_{\ell=1}^L F_\ell^\tr s_\ell\right)^\tr \dx \\
    &+ \textstyle \left((I_{T-1} \kron B^\tr) \mu \re{-} \sum_{\ell=1}^L G_\ell^\tr s_\ell\right)^\tr \du. \nonumber
\end{align}
Expressing the $(\dx, \du)$-constant terms of $\Phi$ as $Q$ with
\begin{equation}
    \textstyle Q = -q + \mu^\tr h^0\re{+} \sum_{\ell=1}^L \tau_\ell,
\end{equation}
the supremal value of $\Phi$ w.r.t. $\dx, \du$ has a value of
\begin{equation}
    \sup_{\dx, \du} \Phi = \begin{cases} 
    Q & \textrm{if} \  \re{0=}\Xi^\tr \mu \re{-}\sum_{\ell=1}^L F_\ell^\tr s_\ell \\
    & \textrm{if} \ \re{0=}(I_{T-1} \kron B^\tr) \mu \re{-} \sum_{\ell=1}^L G_\ell^\tr s_\ell \\
    \infty & \textrm{otherwise}.\end{cases}
\end{equation}

\begin{thm}
    Problem \eqref{eq:q_feas} will have the same feasibility (or infeasibility) status as the following program:
    \begin{subequations}
    \label{eq:q_altern}
        \begin{align}
            \find_{\mu, s, \tau} \ \ \textrm{\re{s.t.}} \ & \forall (A, B) \in \pc: \\
            & \qquad q \re{-} \textstyle \sum_{\ell=1}^L \tau_\ell - \mu^\tr h^0 \geq 0 \\
            & \textstyle \qquad \Xi^\tr \mu \re{-} \sum_{\ell=1}^L F_\ell^\tr s_\ell = 0  \\
            & \qquad \textstyle (I_{T-1} \kron B^\tr) \mu \re{-} \sum_{\ell=1}^L G_\ell^\tr s_\ell =0 \\
            & \textrm{$(\mu, s, \tau)$ from \eqref{eq:dual_must}.}
        \end{align}
    \end{subequations}
\end{thm}
\begin{proof}
    This relationship holds using the convex-duality based Theorem of Alternatives from Section 5.8 of \cite{boyd2004convex}, given that all description expressions in $q$ and $\bpc$ are affine in the uncertain terms $(\dx, \du)$.
\end{proof}

\begin{thm}
    If $\pc$ is compact, then the multipliers $(\mu, s, \tau)$ can be chosen to be polynomial functions of $(A, B)$.
\end{thm}
\begin{proof}
    The proof is omitted due to page limitations, but follows from Theorem 3.3 of \cite{miller2023analysis} (generalizing Theorems 4.4 and 4.5 of \cite{miller2022eiv_short} to the conically constrained case).
\end{proof}

\begin{rmk}
    In the case of probabilistic noise set from \eqref{eq:chisq} with $L = 2T - 1$, the multipliers $(s, \tau)$ can be partitioned into $\forall t \in 1..T: \ (s^x_t, \tau^x_t)$ for \eqref{eq:chisq_x} and $\forall t \in 1..T-1: \ (s^u_t, \tau^u_t)$ for \eqref{eq:chisq_u}.
    The certificate \eqref{eq:q_altern} can then be expressed as 
        \begin{subequations}
    \label{eq:q_altern_chisq}
        \begin{align}
            \find_{\mu, s, \tau} \  \textrm{\re{s.t.}} \ & \forall (A, B) \in \pc: \\
            & \qquad q \re{-} \textstyle \sum_{\ell=1}^T \sqrt{\varepsilon(\delta_x; n)} \tau^x_\ell  + \textstyle \sum_{\ell=1}^{T-1} \sqrt{\varepsilon(\delta_u; m)} \tau^u_\ell \nonumber\\
            &\qquad \qquad - \mu^\tr h^0 \geq 0 \\
            & \qquad\Sigma^{-1/2}_x s_1^x = \re{A^\tr \mu_1} \\
    &\forall t=2..T-1: \  \Sigma^{-1/2}_x s_t^x = \re{A^\tr \mu_t - } \mu_{t-1}  & & \\
    &\qquad\Sigma^{-1/2}_x s_T^x =  \re{-}\mu_{T-1} \\
            & \forall t \in 1..T-1: \  \Sigma_u^{-1/2} s^u_t= \re{B^\tr \mu_t}  \\
            & \textrm{$(\mu, s, \tau)$ from \eqref{eq:dual_must}.}
        \end{align}
    \end{subequations}
\end{rmk}
\section{Truncated Sum-of-Squares Programs}
\label{sec:quad_truncate}
% \urg{The main theory}

This section uses the moment-\ac{SOS} hierarchy of \acp{SDP} to discretize the infinite-dimensional \ac{LP} in \eqref{eq:super_full} into finite-dimensional convex optimization problems that are more amenable to computation. This discretization will be performed by \ac{SOS}-matrix truncations.

\subsection{Quadratically-Bounded Truncations}

% \urg{SOS truncations}

Program \eqref{eq:super_full} has $2f^2 + f$ infinite-dimensional \ac{LP} constraints posed over $(A, B, \dx, \du)$. This subsection applies a degree $2k$ \ac{SOS} tightening to the constraints in program \eqref{eq:super_full}. In each case, the matrix $M$ is restricted to a polynomial $M \in (\R[A, B]_\re{{\leq 2k}})^{n \times n}.$
The remainder of this section will analyze the computational complexity of an \ac{SOS} tightening on a single infinite-dimensional constraint from \eqref{eq:super_full} (represented as $q(A, B) \geq 0$ from \eqref{eq:q_feas}). Complexity will be compared according to the \ac{PSD} Gram matrix of maximal size.

\subsubsection{Full Program}

The full program applies a scalar \ac{WSOS} constraint $q \in \Sigma^1[\bpc]$  over $(A, B, \dx, \du)$. The size of the maximal Gram matrix  for each \eqref{eq:q_feas} restriction is 
\begin{equation}
   \textstyle  p_F = \binom{n(n+T) + m(n+T-1)+k}{k}.\label{eq:size_full}
\end{equation}

% \urg{\ac{WSOS} programs for $\Sigma[\bpc]$ involving $(A, B, \dx, \du)$ }

\subsubsection{Alternatives (Dense)}

% \urg{Theorem for convergence? Presentation of set $\Pi$?}

% The \ac{SOC} constraint from 

The \ac{SOC} constraint in \eqref{eq:soc_cond} can also be expressed by a PSD constraint \cite{alizadeh2003second},
\begin{align}
    (s_\ell, \tau_\ell) \in \ls^{n_\ell} & \Leftrightarrow  \begin{bmatrix}
        \tau_\ell  & s_{\ell}^\re{\tr} \\ s_{\ell} & \tau_\ell \re{I_{n_\ell}}
    \end{bmatrix} \in \psd^{n_\ell+1}_+. \label{eq:soc_sdp_arrow}
\end{align}

In order to prove convergence of the Alternative truncations as the degree $k \rightarrow \infty$, we must assume that there exists a known Archimedean set $\Pi(A, B)$ such that $\Pi \supseteq \pc$. Such a $\Pi$ might be known from norm or Lipschitz bounds on $(A, B)$, or other similar knowledge on reasonable plant behavior. If $\Pi$ is a-priori unknown, then we will use the \ac{WSOS} symbol $\Sigma^r[\Pi]$ to refer to the \ac{SOS} set $\Sigma^r[x]$. 
At the degree-$k$ truncation, the multipliers from \eqref{eq:dual_must} can be chosen:
\vspace{-0.5cm}
\begin{subequations} %TODO: careful of vspace
\label{eq:multipliers_arrow}
\begin{align}
    \mu &\in (\R[\Pi]_{\leq 2k - 1})^{nT} \\
    s_\ell & \in (\R[\Pi]_{\leq 2k})^{n_\ell} & & \forall \ell \in 1..L \\
    \tau_\ell & \in \R[\Pi]_{\leq 2k} & & \forall \ell \in 1..L.
\end{align}
These multipliers are required to satisfy: 
\begin{align}
& \begin{bmatrix}
        \tau_\ell(A, B)  & s_{\ell}(A, B)^\re{\tr} \\ s_{\ell}(A, B) & \tau_\ell(A, B)\re{I_n}
    \end{bmatrix}  \in \Sigma^{n_\ell+1}[\Pi]_{\leq 2k} &  \forall \ell \in 1..L. \label{eq:soc_sos_large}
\end{align}
\end{subequations}

The matrix \ac{WSOS} constraint in \eqref{eq:soc_sos_large} has \iac{PSD} Gram matrix of maximal size 
\begin{equation}
    \textstyle p_A^\ell = (n^\ell+1)\binom{n(n+m)+k}{k}.\label{eq:size_altern}
\end{equation}
The `dense' nomenclature for this approach will refer to the imposition that \eqref{eq:soc_sos_large} is \ac{WSOS} over a matrix of size $n_\ell+1$ for each $\ell=1..L.$
% Positive-definiteness of the right-hand matrix of \eqref{eq:soc_sdp_arrow} can be ensured by requiring its membership in an $(n+1)$-dimension 

\subsubsection{Alternatives (Sparse)}

The \ac{SOC} constraint in \eqref{eq:soc_sdp_arrow} can be decomposed into $2 \times 2$ blocks as in \cite{alizadeh2003second}:
\begin{align}
\nonumber
       (s_\ell, \tau_\ell) \in \ls^{n_\ell} & \Leftrightarrow  \exists z_{i\ell}: \ \begin{bmatrix}
        \tau_\ell & s_{i \ell} \\ s_{i \ell} & z_{i\ell} 
    \end{bmatrix} \in \psd^2_+, \quad \tau_\ell = \textstyle \sum_{i=1}^n z_{i \ell}.
\end{align}

The multipliers in \eqref{eq:dual_must} can be restricted to
\begin{subequations}
\label{eq:multipliers_sparse}
\begin{align}
    \mu &\in (\R[\Pi]_{\leq 2k - 1})^{nT} \\
    s_\ell, \re{\ z_\ell} & \in (\R[\Pi]_{\leq 2k})^{n_\ell} & \forall \ell \in 1..L
\end{align}
subject to the constraint 
\begin{align}
\begin{bmatrix}
       \textstyle  \re{\sum_{i=1}^{n_\ell} z_{i \ell}}(A, B) & s_{i \ell}(A, B) \\ s_{i \ell}(A, B) & z_{i \ell}(A, B) 
    \end{bmatrix} \in \Sigma^{2}[\Pi]_{\re{\leq 2k}} & & \forall \ell=1..L. \label{eq:soc_sos_small}
\end{align}
\end{subequations}
The Gram matrices from \eqref{eq:soc_sos_small} have a reduced size of 
\begin{equation}
   \textstyle  p_B = 2 \binom{n(n+m)+k}{k}. \label{eq:size_block}
\end{equation}
% as compared to $p_A^\ell = (n^\ell+1) \binom{n(n+m)+k}{k}$ from \eqref{eq:soc_sos_large}.

\begin{rmk}
The cone constraints \eqref{eq:soc_sos_large} and \eqref{eq:soc_sos_small} are not necessarily equivalent at finite degree $2k$, although they will describe the same set as $k \rightarrow \infty$ given that $(s, \tau)$ are optimization variables. Refer to \cite{papp2013semidefinite} for more details on the relationship between matrix \ac{SOC} cone representations.
    % The set of polynomials $(s, \tau)$ satisfying 
\end{rmk}

\subsection{Computational Complexity}

Computational complexity of the Full, Alternatives (Dense), and Alternatives (Decomposed) schemes will be judged comparing the sizes and multiplicities of the largest \ac{PSD} constraint in any $\bpc$-nonnegativity constraint. As a reminder, the  superstabilizing program in \eqref{eq:super_full} has $2n^2+n$ such $\bpc$-nonnegativity linear inequality constraints.

For each $\bpc$-nonnegativity constraint, the Full program requires 1 block of maximal size $p_F$ from \eqref{eq:size_full}. The Alternatives programs require $L+1$ \ac{PSD} blocks each, in which the Dense program has block sizes of $p_A^\ell$ from \eqref{eq:size_altern} for $\ell=1..L$ (and a scalar block of size $\binom{n(n+m)+k}{k}$). The sparse program has all $L$ blocks of size $p_B$ from \eqref{eq:size_block} along with a similar scalar block of size $\binom{n(n+m)+k}{k}$.
Table \ref{tab:psd_size} reports the size of the largest \ac{PSD} matrix constraint for the three approaches under $n=2, m=2, k=2, T = 12$. The considered quadratically-bounding constraints all have $L=2$, such as in the elementwise-norm  $\norm{\dx_t}_2 \leq \epsilon_x$ and $\norm{\du_t}_2 \leq \epsilon_u$ from Section \ref{sec:elem_norm}.
% The size  $r \binom{n+k}{k}$ will be used to compare computational complexity of \ac{SOS} truncation algorithms. 
We first note that the per-iteration complexity of solving an \ac{SDP} using an interior point method (with \ac{PSD} size $N$ having $M$ affine constraints) is $O(N^3 M + N^2 M^2)$ \cite{alizadeh1995interior}. 
% Verification of \ac{SOS} membership in \eqref{eq:sos_mat} at a fixed degree $k$ involves an \ac{SDP} with $N = r \binom{n+k}{k}$ and $M = \frac{r(r+1)}{2} \binom{n+k}{k}$. The scaling of the per-iteration complexity is generally dominated by the $M^2 N^2$ term, which grows as $r^6 n^{6k}$ or $r^6 k^{4n}.$ 
% While the specifics of parameter scaling for nonsymmetric interior point \ac{SOS} solvers are different (lacking the affine constraint embedding) \cite{papp2019sum}, the maximal \ac{PSD} size remains a valuable way to compare predicted performance. 
In the context of Table \ref{tab:psd_size}, the size of $N = p_F = 1540$ is intractably large for current interior-point methods. The `\re{Multiplicity}' scaling for Dense and Sparse causes computational complexity to grow linearly, while the increasing `Size' parameter causes polynomial growth in scaling.
% \urg{Tianyu's Tables to show results?}
\begin{table}[h]
    \centering
        \caption{Size of \ac{PSD} Variables for $\bpc$-nonnegativity}
    \label{tab:psd_size}
    % \begin{adjustbox}{width=0.3\textwidth}
    \begin{tabular}{ |l|l|c| }
 \hline
  & Size & \re{Multiplicity} \\ \hline
 Full & {$p_F = 1540$} & {1}  \\ \hline
 Dense & {$p_A^\ell = 135$} & {46} \\ \hline
Sparse & {$p_B = 90$} & {46} \\ \hline
    \end{tabular}
    % \end{adjustbox}
\end{table}
\section{Extended Superstability}
\label{sec:extended_super}
\re{
The superstabilization method considered in Sections \ref{sec:quad_lp} and \ref{sec:quad_truncate} rely on the use of a previously given and fixed $W$ matrix. The framework of Extended Superstability \cite{polyak2004extended} allows for $W$ to 
be chosen as a positive-definite $n$-dimensional diagonal matrix, in which this diagonal-restricted $W$ may be searched over in optimization and is not fixed in advance. The resultant common Lyapunov function $\norm{W x}_\infty$ therefore has hyper-rectangular sublevel sets.
Letting $v \in \R^n$, $v > 0$ be a positive vector with matrix $W = \diag{v}^{-1},$ and $x_{t+1} = A x_t$ be a dynamical system, the $W$-superstabilization condition $\norm{W A W^+}_\infty < 1$ may be expressed as 
\begin{subequations}
\begin{align}
    \forall i = 1..n, j = 1..n& & [A \ \textrm{diag}(v)]_{ij} \leq M_{ij}\\  
    \forall i = 1..n& &  \textstyle \sum_{j=1}^n M_{ij} < v_i. 
\end{align}
\end{subequations}

The quadratically-bounded \ac{EIV} extended superstabilization task involves the following optimization problem with variables $ M(A, B): \pc \rightarrow \R^{n \times n}, S \in \R^{m \times n}, v \in \R^n:$
\begin{subequations}
\label{eq:ext_superstable}
\begin{align}
\find_{M, \ S, \ v} \ \textrm{s.t. } \  & \forall (A, B) \in \pc: \\
& \quad \forall i \in 1..n: \label{eq:ext_superstable_strict}\\
&\quad \textstyle \qquad v_i > 0,  \   \sum_{j=1}^n M_{ij}(A, B) < v_i \nonumber  \\
&\quad  \forall i \in 1..n, \ j \in 1..n:  \label{eq:ext_superstable_nonstrict} \\
    &\quad \qquad \textstyle \abs{A_{ij}v_j + \sum_{\ell=1}^{\rev{m}} B_{i\ell}S_{\ell j}} \leq M_{ij}(A, B) & & \nonumber\\
    &M: \pc \rightarrow \R^{n \times n}, \, S \in \R^{m \times n}, v \in \R^n.
\end{align}
\end{subequations}
If Program \eqref{eq:ext_superstable} is feasible, then the controller $K$  recovered by $K = S W = S (\diag{v})^{-1}$ is guaranteed to $W$-superstabilize all systems in $\pc$. All three methods (Full, Dense, Sparse) for \ac{WSOS} truncation can be employed to formulate order-$k$ versions of Program \eqref{eq:ext_superstable} (w.r.t. a Psatz in $(A, B, \dx, \du)$ for Full or $(A, B)$ for Alternatives) by choosing  $M \in (\R[A, B]_{\leq 2k})^{n \times n}$. 
% While the set of extended superstable systems is a superset of the class of superstable systems $(W=I)$, 

}
\section{Numerical Example}

\label{sec:examples}

Code to generate the following experiment is available at  \url{https://github.com/Jarmill/eiv_quad}. The \acp{SDP} deriving from the \ac{SOS} programs were synthesized through \texttt{JuMP} \cite{Lubin2023jump} and solved using Mosek 10.1 \cite{mosek92}. 

% \urg{I have 2d experiments with superstabilization. Not extended superstabilization.}

\subsection{White Noise}

This example involves a 2-state 2-input ground-truth plant:
% The first example involves the following 2-state 2-input ground-truth plant:
\begin{align}
    A_{\star} = \begin{bmatrix}
        0.6863  &  0.3968 \\
    \re{-}0.3456  &  1.0388
    \end{bmatrix}, \ B_{\star} = \begin{bmatrix}
        0.4170 & 0 \\ 0.7203 & \re{-}0.3023
    \end{bmatrix}.\label{eq:sys_white}
\end{align}

Data $\dc$ is collected from an execution of \eqref{eq:sys_white} over a time-horizon of $T=\re{13}$. In the data collection, the \ac{EIV} noise $\dx,\du$ are i.i.d. distributed according to $\dx_t \re{\sim} \mathcal{N}(\0, 0.0\re{3}^2 I_2), \du_t \re{\sim} \mathcal{N}(\0, 0.0\re{25}^2 I_2)$. It is desired to create \re{an extended} superstabilizing controller that will succeed in regulating the ground-truth system with joint probability $P{_\textrm{joint}} = 95\%.$ The per-noise probability is chosen as $\delta_x = \delta_u = (0.95)^{1/(2T-1)}=0.9981$. The $(\delta_x, \delta_u)$ chance constraint is modeled as $\norm{\dx_t}_2 \leq 0.0\re{3} \epsilon(\delta_x; 2) = 0.\re{1056}$ and $\norm{\du_t}_2 \leq 0.0\re{25} \epsilon(\delta_u; 2)=0.\re{08796}$ for each $t$.

\re{Extended} superstabilizing control is performed to minimize \re{$\lambda$ such that $\forall i: \lambda\geq v_i$}.
% under a decay-bound objective $\lambda^* = \min_{\lambda, K} \lambda \geq \norm{A+BK}_\infty \ \forall (A, B) \in \pc$. 
The $k=1$ truncation of \re{\eqref{eq:ext_superstable}} using the \re{dense Alternatives method} \eqref{eq:multipliers_arrow} ($p_A = 27$) returns
\begin{align}
K_{\textrm{dense}} &= \re{\begin{bmatrix}
 -0.9000 &   -0.8807\\
  0.1564  &  0.3679
    \end{bmatrix} }& & 
    \re{v = \begin{bmatrix}
         0.5519\\
 1.4481
    \end{bmatrix}.}
    % \lambda^*_{\textrm{dense}} =  0.9907.
    \label{eq:controller_white}
\end{align}
% \begin{align}
% K_{\textrm{dense}} &= \begin{bmatrix}
%  -1.51402 & -0.67711\\
%   1.4894  & -1.80697 
%     \end{bmatrix} & & \lambda^*_{\textrm{dense}} =  0.4372 \label{eq:controller_white} \\
%     K_{\textrm{sparse}} &= \begin{bmatrix}
%          -0.34260  &  -0.49678\\
%        0.0037668  & -1.1134
%     \end{bmatrix} & & \lambda^*_{\textrm{sparse}} =  0.9972. \nonumber
% \end{align}

% Worst-case superstabilization is certified by the value $\lambda^* < 1$. 

% The ground-truth system \eqref{eq:sys_white} has a decay-bound of $\lambda_{\star} = \norm{A_{\star} + B_{\star} K_{\textrm{dense}}}_\infty = 0.8219$ with respect to the controller $K$ in \eqref{eq:controller_white}.
% The Full program ($p_F = 62$) and Sparse ($p_B = 18$) fail to perform superstabilization at $k=1$ because the returned bound satisfies $\lambda^*_{\textrm{full}} = 3.1749 \times 10^{3}, \lambda^*_{\textrm{sparse}} = 2.4173 \geq 1$. 
\re{The Full \eqref{eq:super_full}, Dense \eqref{eq:multipliers_arrow}, and Sparse \eqref{eq:multipliers_sparse} superstabilization $(W=I)$ programs are all infeasible for $k=1$.}
Attempting execution of the $k=2$ tightening for the Full program ($p_F = \binom{64}{2} = 2016$) results in out-of-memory errors in Mosek.
For this specific example, the algorithms from Theorems 1 and 2 of \cite{bisoffi2024controller} both fail to find a common quadratically stabilizing controller.
%on (matrix-ellipsoid-based over-approximations of) $\pc$.

\subsection{Monte Carlo Test}

% \urg{
% Output report: (to be put into text)
% julia: sum(report\_ss)
% 48

% julia: sum(report\_ss\_sparse)
% 42

% julia: sum(report\_qmi)
% 61

% julia: sum(report\_ess)
% 71

% julia: sum(report\_ess\_sparse)
% 62

% julia: sum(report\_ess .* report\_qmi)
% 38

% julia: sum(report\_ss .* report\_qmi)
% 31}

\rev{This second example involves a Monte Carlo test for robust stabilization of 300 randomly generated 2-state 2-input ground-truth plants. Elements of the plant matrices $(A, B)$ are each i.i.d. drawn from a unit normal distribution. The noise $\Delta x_t$ and $\Delta u_t$ are drawn i.i.d. uniformly from unit $L_2$-balls of radius $\epsilon_x = 0.225$ and $\epsilon_u = 0.1$ respectively for a time horizon of $T=14$. At $k=1$, dense and sparse superstable \ac{SOS} restrictions stabilized 48 and 42 systems respectively. Dense and sparse extended superstabilization $k=1$ \ac{SOS} restrictions stabilized 71 and 62 systems respectively. Theorem 1 of \cite{bisoffi2024controller} was infeasible at each instance. Theorem 2 of \cite{bisoffi2024controller} stabilized 61 systems, with an overlap of 31 stabilized systems with the $k=1$ dense superstability, and 38 systems with the $k=1$ dense extended superstability.
% For this system, the 
 % order $k=1$ superstabilization dense \eqref{eq:multipliers_arrow} stabilizes 2/200 system, and Theorem 2 from \cite{bisoffi2024controller} stabilizes 3/200 systems. 
 % Specifically, the $k=1$ superstabilization produces an infeasibility certificate 96/200 times, and computes $\lambda^*:1$ on the remainder. 
 % Extended superstability \eqref{eq:ext_superstable} at a $k=1$ dense \ac{SOS} truncation stabilizes the system \urg{?}/200 times, with numerical errors present \urg{?}/200 times.
 }

% Theorem 2 of \cite{bisoffi2024controller} and the Sparse $(k=1)$ \eqref{eq:multipliers_sparse} succeed at finding stabilizing controllers with $P{_\textrm{joint}}=90\%,$ but $(A_\star, B_\star)$ is not a member of the $90\%$-chance consistency set $\pc$.
% in \eqref{eq:multipliers_sparse}.

% \subsection{\urg{Another Experiment}}

% \urg{Put together another experiment}
% and the resultant chi-square radius is $\epsilon(\delta_x, 2) = $
% The noise values $\dx,\du$ are each i.i.d. distributed according to a unit normal distribution (e.g., $\dx_t \in \mathcal{N}(\0, I_2)$).

\section{Conclusion}

\label{sec:conclusion}

This paper presented a solution approach for data-driven  superstabilization in the quadratically-bounded \ac{EIV} setting. The $W$-superstabilization problem was formulated as an infinite-dimensional linear program, and was discretized using \ac{SOS}-matrix methods. A Theorem of Alternative was used to eliminate the $(\dx, \du)$ noise terms, resulting in matrix \ac{SOS} constraints with lower computational complexity.

Future work involves 
\re{reducing conservatism of \ac{EIV}-aware control methods and incorporating streaming data for \ac{EIV}-tolerant model predictive control.}
%Future work involves 
%decreasing the runtime %solutions to \ac{EIV} superstabilization, 
%\re{reducing conservatism of \ac{EIV}-aware control methods, and incorporating streaming data for \ac{EIV}-tolerant model predictive control.}
% . Other aspects include incorporating streaming data for online superstabilization, which could include integrating the \ac{EIV} method into model predictive control. 
% More directions include integrating additional noise sets into the \ac{EIV} robustification framework.

% \urg{Conclude the paper. Summarize the problem and important findings. Add some extensions and future work.}

% An extended Arxiv version of this paper is available at \urg{[Arxiv link goes here] what will the arxiv version have that this work does not?}.

% \input{sections/acknowledgements}

% \input{sections/appendix}

% %%%%%%%%%%%%%%%%%%%%%%%%%%%%%%%%%%%%%%%%%%%%%%%%%%%%%%%%%%%%%%%%%%%%%%%%%%%%%%%%
% \section{Acknowledgements}

% The authors thank Milan Korda for his discussions about occupation measures and time-varying uncertainty.

\bibliographystyle{IEEEtran}
\bibliography{references.bib}

\begin{thebibliography}{10}
\providecommand{\url}[1]{#1}
\csname url@samestyle\endcsname
\providecommand{\newblock}{\relax}
\providecommand{\bibinfo}[2]{#2}
\providecommand{\BIBentrySTDinterwordspacing}{\spaceskip=0pt\relax}
\providecommand{\BIBentryALTinterwordstretchfactor}{4}
\providecommand{\BIBentryALTinterwordspacing}{\spaceskip=\fontdimen2\font plus
\BIBentryALTinterwordstretchfactor\fontdimen3\font minus \fontdimen4\font\relax}
\providecommand{\BIBforeignlanguage}[2]{{%
\expandafter\ifx\csname l@#1\endcsname\relax
\typeout{** WARNING: IEEEtran.bst: No hyphenation pattern has been}%
\typeout{** loaded for the language `#1'. Using the pattern for}%
\typeout{** the default language instead.}%
\else
\language=\csname l@#1\endcsname
\fi
#2}}
\providecommand{\BIBdecl}{\relax}
\BIBdecl

\bibitem{soderstrom2018errors}
T.~S{\"o}derstr{\"o}m, \emph{{Errors-in-Variables Methods in System Identification}}.\hskip 1em plus 0.5em minus 0.4em\relax Springer, 2018.

\bibitem{ben2009robust}
A.~Ben-Tal, L.~El~Ghaoui, and A.~Nemirovski, \emph{Robust Optimization}.\hskip 1em plus 0.5em minus 0.4em\relax Princeton University Press, 2009, vol.~28.

\bibitem{MARTIN2023gaussian}
T.~Martin, T.~B. Schön, and F.~Allgöwer, ``Gaussian inference for data-driven state-feedback design of nonlinear systems,'' \emph{IFAC-PapersOnLine}, vol.~56, no.~2, pp. 4796--4803, 2023, 22nd IFAC World Congress.

\bibitem{HOU20133}
Z.-S. Hou and Z.~Wang, ``{From model-based control to data-driven control: Survey, classification and perspective},'' \emph{Information Sciences}, vol. 235, pp. 3--35, 2013, data-based Control, Decision, Scheduling and Fault Diagnostics.

\bibitem{martin2023guarantees}
T.~Martin, T.~B. Sch{\"o}n, and F.~Allg{\"o}wer, ``Guarantees for data-driven control of nonlinear systems using semidefinite programming: A survey,'' \emph{Annual Reviews in Control}, p. 100911, 2023.

\bibitem{waarde2020noisy}
H.~J. van Waarde, M.~K. Camlibel, and M.~Mesbahi, ``{From Noisy Data to Feedback Controllers: Nonconservative Design via a {M}atrix {S}-{L}emma},'' \emph{IEEE Trans. Automat. Contr.}, 2020.

\bibitem{bisoffi2021trade}
A.~Bisoffi, C.~De~Persis, and P.~Tesi, ``Trade-offs in learning controllers from noisy data,'' \emph{Syst. Control Lett.}, vol. 154, p. 104985, 2021.

\bibitem{dai2018data}
T.~Dai and M.~Sznaier, ``{Data Driven Robust Superstable Control of Switched Systems},'' \emph{IFAC-PapersOnLine}, vol.~51, no.~25, pp. 402--408, 2018.

\bibitem{miller2023ddcpos}
J.~Miller, T.~Dai, M.~Sznaier, and B.~Shafai, ``Data-driven control of positive linear systems using linear programming,'' in \emph{2023 62nd IEEE Conference on Decision and Control (CDC)}.\hskip 1em plus 0.5em minus 0.4em\relax IEEE, 2023, pp. 1588--1594.

\bibitem{dai2020semi}
T.~Dai and M.~Sznaier, ``A semi-algebraic optimization approach to data-driven control of continuous-time nonlinear systems,'' \emph{IEEE Control Systems Letters}, vol.~5, no.~2, pp. 487--492, 2020.

\bibitem{martin2021data}
T.~Martin and F.~Allg{\"o}wer, ``{Data-driven system analysis of nonlinear systems using polynomial approximation},'' \emph{arXiv preprint arXiv:2108.11298}, 2021.

\bibitem{campi2002virtual}
M.~C. Campi, A.~Lecchini, and S.~M. Savaresi, ``{Virtual reference feedback tuning: a direct method for the design of feedback controllers},'' \emph{Automatica}, vol.~38, no.~8, pp. 1337--1346, 2002.

\bibitem{coulson2019data}
J.~Coulson, J.~Lygeros, and F.~D{\"o}rfler, ``{Data-enabled predictive control: In the shallows of the DeePC},'' in \emph{2019 18th European Control Conference (ECC)}.\hskip 1em plus 0.5em minus 0.4em\relax IEEE, 2019, pp. 307--312.

\bibitem{SZNAIER19963550}
M.~Sznaier, R.~Suárez, S.~Miani, and J.~Alvarez-Ramírez, ``{Optimal $l_\infty$ Disturbance Rejection and Global Stabilization of Linear Systems with Saturating Control},'' \emph{IFAC Proceedings Volumes}, vol.~29, no.~1, pp. 3550--3555, 1996, 13th World Congress of IFAC, 1996, San Francisco USA, 30 June - 5 July.

\bibitem{polyak2002superstable}
B.~T. Polyak and P.~S. Shcherbakov, ``{Superstable Linear Control Systems. I. Analysis},'' \emph{Automation and Remote Control}, vol.~63, no.~8, pp. 1239--1254, 2002.

\bibitem{miller2022eiv_short}
J.~Miller, T.~Dai, and M.~Sznaier, ``{Data-Driven Superstabilizing Control of Error-in-Variables Discrete-Time Linear Systems},'' in \emph{2022 IEEE 61st Conference on Decision and Control (CDC)}.\hskip 1em plus 0.5em minus 0.4em\relax IEEE, 2022, pp. 4924--4929.

\bibitem{miller2023eivarx}
------, ``{Superstabilizing Control of Discrete-Time ARX Models under Error in Variables},'' \emph{IFAC-PapersOnLine}, vol.~56, no.~2, pp. 2444--2449, 2023, 22nd IFAC World Congress.

\bibitem{boyd2004convex}
S.~Boyd and L.~Vandenberghe, \emph{{Convex Optimization}}.\hskip 1em plus 0.5em minus 0.4em\relax Cambridge university press, 2004.

\bibitem{bisoffi2024controller}
A.~Bisoffi, L.~Li, C.~D. Persis, and N.~Monshizadeh, ``Controller synthesis for input-state data with measurement errors,'' 2024, arxiv:2402.04157.

\bibitem{polyak2004extended}
B.~T. Polyak, ``Extended superstability in control theory,'' \emph{Automation and Remote Control}, vol.~65, no.~4, pp. 567--576, 2004.

\bibitem{hsu2012subgaussian}
D.~Hsu, S.~Kakade, and T.~Zhang, ``{A tail inequality for quadratic forms of subgaussian random vectors},'' \emph{Electronic Communications in Probability}, vol.~17, no. none, pp. 1 -- 6, 2012.

\bibitem{lasserre2009moments}
J.~B. Lasserre, \emph{{Moments, Positive Polynomials And Their Applications}}, ser. Imperial College Press Optimization Series.\hskip 1em plus 0.5em minus 0.4em\relax World Scientific Publishing Company, 2009.

\bibitem{scherer2006matrix}
C.~W. Scherer and C.~W. Hol, ``{Matrix Sum-of-Squares Relaxations for Robust Semi-Definite Programs},'' \emph{Mathematical Programming}, vol. 107, no.~1, pp. 189--211, 2006.

\bibitem{cerone1993feasible}
V.~Cerone, ``Feasible parameter set for linear models with bounded errors in all variables,'' \emph{Automatica}, vol.~29, no.~6, pp. 1551--1555, 1993.

\bibitem{miller2022eivlong}
J.~Miller, T.~Dai, and M.~Sznaier, ``{Data-Driven Stabilizing and Robust Control of Discrete-Time Linear Systems with Error in Variables},'' 2023, arXiv:2210.13430.

\bibitem{miller2023analysis}
J.~Miller and M.~Sznaier, ``{Analysis and Control of Input-Affine Dynamical Systems using Infinite-Dimensional Robust Counterparts},'' 2023, arXiv:2112.14838.

\bibitem{alizadeh2003second}
F.~Alizadeh and D.~Goldfarb, ``Second-order cone programming,'' \emph{Mathematical Programming}, vol.~95, no.~1, pp. 3--51, 2003.

\bibitem{papp2013semidefinite}
D.~Papp and F.~Alizadeh, ``{Semidefinite Characterization of Sum-of-Squares Cones in Algebras},'' \emph{SIAM Journal on Optimization}, vol.~23, no.~3, pp. 1398--1423, 2013.

\bibitem{alizadeh1995interior}
F.~Alizadeh, ``{Interior Point Methods in Semidefinite Programming with Applications to Combinatorial Optimization},'' \emph{SIAM J. Optim.}, vol.~5, no.~1, pp. 13--51, 1995.

\bibitem{Lubin2023jump}
M.~Lubin, O.~Dowson, J.~D. Garcia, J.~Huchette, B.~Legat, and J.~P. Vielma, ``{JuMP 1.0: Recent improvements to a modeling language for mathematical optimization},'' \emph{Mathematical Programming Computation}, 2023.

\bibitem{mosek92}
{MOSEK ApS}, \emph{The MOSEK optimization toolbox for MATLAB manual. Version 10.1.}, 2023.

\end{thebibliography}

\end{document}